\documentclass[11pt,reqno,a4paper]{amsart}
\usepackage[latin1]{inputenc}
\usepackage[latin1]{inputenc}
\usepackage{a4wide}
\usepackage{fancyhdr}
\usepackage{amsfonts,amssymb,amsmath,amsthm,latexsym,indentfirst}
\usepackage{epsfig}
\usepackage{enumerate}
\usepackage{mathtools}
\usepackage{amsfonts}
\usepackage{latexsym}
\usepackage{amssymb}
\usepackage{csquotes} 
\usepackage{amsmath}
\usepackage{textcomp}
\usepackage{xcolor}
\usepackage[pagewise]{lineno}
\usepackage{graphicx}
\usepackage{color}
\usepackage{hyperref}
\hypersetup{colorlinks=true,citecolor=blue,urlcolor=blue}

\numberwithin{equation}{section}
\newtheorem{Theorem}{Theorem}

\newtheorem{theorem}{Theorem}[section]

\newtheorem{lemma}[theorem]{Lemma}

\newcommand{\N}[1][]{\ensuremath{{\mathbb{N}^{#1}} }}

\newcommand{\R}[1][]{\ensuremath{{\mathbb{R}^{#1}} }}

\newcommand{\dis}{\displaystyle}

\newcommand{\mc}{\mathcal}

\newcommand{\p}{\partial}
\DeclareMathOperator\supp{supp}

\title[On the stabilization of $L^2$ and $H^1$ norms for the Z-K eq. with damping]{On the stabilization of $L^2$ and $H^1$ norms for the Zakharov-Kuznetsov equation with damping}

\author[Mykael Cardoso, Roger P. de Moura and Gleison N. Santos]{Mykael A. Cardoso, Roger P. de Moura and Gleison N. Santos}

\address{Universidade Federal do Piau\'i, Campus Universit\'ario Ministro Petrônio Portella, Ininga, 64049-550, Teresina, Piau\'i
, Brazil}
\email{mykael@ufpi.edu.br}

\address{Universidade Federal do Piau\'i, Campus Universit\'ario Ministro Petrônio Portella, Ininga, 64049-550, Teresina, Piau\'i
, Brazil}
\email{mourapr@ufpi.edu.br}

\address{Universidade Federal do Piau\'i, Campus Universit\'ario Ministro Petrônio Portella, Ininga, 64049-550, Teresina, Piau\'i
, Brazil}
\email{gleison@ufpi.edu.br}



\begin{document}

\begin{abstract}
In this paper we establish exponential decay results for solutions of the damped $n$-dimensional Zakharov--Kuznetsov equation for $2 \le n \le 3$. 
More precisely, we prove the exponential decay of the $L^2(\mathbb{R}^n)$ norm when the damping is localized. 
In addition, when the dissipative mechanism acts on the whole space $\mathbb{R}^n$, we prove the exponential decay of the $H^1(\mathbb{R}^n)$ norm. Our strategy of proof combines a Kato's type smoothing effect, unique continuation and an observability inequality. 
\end{abstract}

\maketitle

\section{Introduction}

The main purpose of this work is to study the exponential decay of the energy for the initial value problem associated with the $n$-dimensional Zakharov--Kuznetsov equation ($n=2,3$)
\begin{equation}\label{eqZK1}
	\partial_t u + \partial_{x_1}\Delta u + u\partial_{x_1}u = 0,
\end{equation}
where $u=u(x,t)$, with $(x,t)\in\mathbb{R}^n\times\mathbb{R}_+$, is a real-valued function.

The Zakharov--Kuznetsov equation \eqref{eqZK1} is a higher-dimensional extension of the Korteweg--de Vries equation describing the propagation of long surface waves. It was formally derived by Zakharov and Kuznetsov \cite{ZK} as a long-wave, small-amplitude limit of the Euler--Poisson system under the ``cold plasma'' approximation in plasma physics. A rigorous justification of this long-wave limit can be found in \cite{LLS}.

Regarding well-posedness results for the initial value problem (IVP) associated with \eqref{eqZK1} in dimension $n=2$, Faminskii \cite{Fa} proved local well-posedness for initial data in $H^m(\R^2)$, $m\ge1$, and global well-posedness in $H^1(\R^2)$. Later, Linares and Pastor \cite{LPas} improved this result by establishing local well-posedness for initial data in $H^s(\R^2)$ with $s>\frac34$. Subsequently, Grünrock and Herr \cite{GH}, and independently Molinet and Pilod \cite{MoPi}, proved local well-posedness for $s>\frac12$. The best result to date is due to Kinoshita \cite{Kinosh}. 
For dimensions $n\geq 3$, we refer to the recent work of Herr and Kinoshita \cite{HK}, where the authors proved local well-posedness for the Cauchy problem associated with \eqref{eqZK1} in $H^s(\R^n)$ for $s>\frac{n-4}{2}$. As a consequence, they also obtained global well-posedness in $L^2(\R^3)$ and in $H^1(\R^4)$, the latter under a smallness condition on the initial data. In particular, their local results imply global well-posedness in $H^1(\mathbb{R}^3)$. The existence of smooth solutions can be obtained using the classical parabolic regularization method, as in Iório and Nunes \cite{IN}.

The following quantities are conserved along the flow
\begin{equation}\label{energy}
	E(u(t))\coloneqq\frac{1}{2}\int_{\mathbb R^n}u^2(x,t)\,dx,
\end{equation}
and
\begin{equation}\label{hamiltonian}
	H(u(t)) \coloneqq \int_{\mathbb{R}^n} \left(|\nabla u (x, t)|^2 - \frac{1}{3}u^3(x, t)\right) dx.
\end{equation}

Although the quantities \eqref{energy} and \eqref{hamiltonian} are conserved, their decay can be enforced by adding a damping term to equation \eqref{eqZK1}. We consider the following initial value problem (IVP):
\begin{equation} \label{eqZK1dampped}
	\left\{
	\begin{array}{lc}
		\partial_t u + \partial_{x_1}\Delta u + u\partial_{x_1}u + a(x)u = 0, & x\in\R^n,\ t\in\mathbb R, \\
		u(x,0)=u_0(x),
	\end{array}
	\right.
\end{equation}
where $a(x)$ is a bounded nonnegative function, which we refer to as a \textit{damping} term. 

It is immediate to observe that if $a(x)$ satisfies
\[
a(x) \geq \alpha_0, \qquad \forall x \in \mathbb{R}^n,
\]
and if $u$ is a smooth solution of \eqref{eqZK1dampped}, then
\begin{equation}\label{energydampped}
	E(u(t)) \leq e^{-2 \alpha_0 t}E(u_0).
\end{equation}

A natural question is whether the same exponential decay holds when $a(x)$ acts only on a localized region $\Omega \subset \mathbb{R}^n$, that is,
\begin{equation}
	a(x) \geq \alpha_0, \qquad \forall x \in\Omega.
\end{equation}

Similar problems have been studied for one-dimensional dispersive equations posed on bounded and unbounded domains (see, e.g., \cite{CavFamNat, DNa, RZ, LPa} and the references therein). For models posed on multidimensional bounded domains, we refer to \cite{CaFNa, MNS, GP, Na}. However, much less is known for dispersive equations posed on bounded domains that are subsets of $\R^n$, and even less when the domain is the whole space $\R^n$, $n \geq 2$, as in the case of the Zakharov--Kuznetsov equation \eqref{eqZK1}. To the best of our knowledge, the only work addressing this problem is that of Natali \cite{Na}. In this direction, we establish the following result.

\begin{theorem}\label{teo1}
	Let $a \in W^{1,\infty}(\mathbb{R}^n)$ be a non-negative function satisfying
	\begin{equation}\label{damping cond}
		a(x) \geq \alpha_0, \qquad \forall \,x \in \mathbb{R}^n \mbox{ with } \ |x_1| > R,
	\end{equation}
	for some $R, \alpha_0 > 0$.  
	Then, given $L > 0$, there exist positive constants $\delta = \delta(L)$ and $C = C(L)$ such that
	\begin{equation}\label{energydamppedtheorem1}
		E(u(t)) \leq C e^{-\delta t} E(u_0),
	\end{equation}
	for any mild solution $u \in X_T^1$ of \eqref{eqZK1dampped} with $u_0 \in H^1(\mathbb{R}^n)$ satisfying $\|u_0\|_{L^2(\mathbb{R}^n)} \leq L$.
\end{theorem}

As in several other works on exponential decay of the energy (cf. \cite{CavFamNat, DNa, RZ, LPa, CaFNa, MNS, Na}), an observability inequality plays a crucial role in proving Theorem \ref{teo1}. In order to obtain it, we will use the following unique continuation result proved in \cite{DMS}.

\begin{Theorem}\label{Tucp}
	Let $u=u(x,t)$ be a smooth solution of the IVP associated with \eqref{eqZK1}, defined on a nondegenerate interval $I=[-T,T]$. Suppose that for some $B>0$,
	\[
	\supp u(t)\subseteq \mc B, \qquad \text{for all } t\in I,
	\]
	where $\mc B:=\underbrace{[-B,B]\times\dots\times[-B,B]}_{n\ \text{times}}$. Then $u\equiv 0$.
\end{Theorem}

We also extend the exponential decay to the $H^1(\R^n)$-norm when the dissipative mechanism acts on the whole space $\R^n$. More precisely, we prove the following result.

\begin{theorem}\label{teo2}
	Let $a \in W^{1,\infty}(\mathbb{R}^n)$ be such that 
	\begin{equation}\label{dc2}
		a(x) \geq \alpha_0 > 0, \qquad \forall\,x \in \mathbb{R}^n.
	\end{equation}
	Then, given $\delta \in (0,\alpha_0)$ and $L>0$, there exists a constant $C=C(\delta,L)>0$ such that
	\begin{equation}\label{H1dec}
		\|u(\cdot,t)\|_{H^1} \leq C e^{-\delta t}\|u_0\|_{H^1},
	\end{equation}
	for any mild solution $u \in X_T^1$ of \eqref{eqZK1dampped} with $u_0 \in H^1(\mathbb{R}^n)$ satisfying $\|u_0\|_{H^1(\mathbb{R}^n)} \leq L$.
\end{theorem}

As far as we know, this is the first result on the exponential decay of solutions in $H^1(\R^n)$ for $n$-dimensional dispersive equations with damping. To prove Theorem \ref{teo2}, we follow the approach of Cavalcanti et al.~\cite{CavFamNat}, where the authors established a similar result for the damped KdV equation. More precisely, they proved the exponential decay of the $H^k(\R)$-norm, $k\in\N$, provided that $a\in W^{k,\infty}(\R)$ satisfies \eqref{dc2}. 

It is important to note that the study of decay rates at the $H^k$-level, $k\ge2$, is not feasible for the Zakharov--Kuznetsov equation. Indeed, unlike the KdV equation, which possesses infinitely many conserved quantities, the ZK equation preserves only the two quantities \eqref{energy} and \eqref{hamiltonian} along the flow.

The remainder of this paper is organized as follows. Section 2 is devoted to notation and auxiliary results. In Section 3 we present a sketch of the proof of well-posedness for the IVP \eqref{eqZK1dampped}, derive a key local smoothing effect, and prove an observability inequality for \eqref{eqZK1dampped}. In Section 4 we prove Theorems \ref{teo1} and \ref{teo2}. Finally, in the appendix we establish the regularity of the solutions to \eqref{eqZK1dampped}.

\section{Notation and Auxiliary Results}

\subsection{Notation}

\begin{enumerate}
	\item Given any positive constants $C$ and $D$, by $C\lesssim D$ we mean that there exists a constant $c>0$ such that
	\[
	C\leq cD.
	\]
	Moreover, $C\sim D$ means that $C\lesssim D$ and $D\lesssim C$.
	
	\item Given $r>0$, we denote
	\[
	Q_r =[-r,r]\times\R^{n-1}= \{ x \in \mathbb{R}^n \; : \; |x_1| \leq r \}.
	\]
	
	\item For domains $\Omega\neq\R^n$, $L^{p}(\Omega)$ norms will be written as $\|\cdot\|_{L^{p}(\Omega)}$, and $L^{p}(\R^n)$ norms will be written simply as $\|\cdot\|_{L^{p}}$ whenever no confusion arises.
	
	For $T>0$ and $1\le p,q \le \infty$, we will use the Lebesgue space-time norms $L^p(\Omega)L^q_T$ and $L^q_TL^p(\Omega)$, defined by
	\[
	\|f\|_{L^{p}(\Omega)L^q_T}
	=
	\Big\|
	\|f(x,t)\|_{L^q([0,T])}
	\Big\|_{L^p(\Omega)},
	\qquad
	\|f\|_{L^{q}_TL^p(\Omega)}
	=
	\Big\|
	\|f(x,t)\|_{L^p(\Omega)}
	\Big\|_{L^q([0,T])}.
	\]
\end{enumerate}

\subsection{Technical Lemmas}

Here we present some technical results that will be crucial for our analysis.

\begin{lemma}[Gagliardo--Nirenberg inequality]
	Let $m,n,j\in\mathbb{N}$, $\alpha\in(\mathbb{Z}_+)^n$ with $|\alpha|=j$, and let $p,q,r\ge1$. Suppose that $\theta$ satisfies
	\[
	\frac{1}{p}-\frac{j}{n}
	=
	\theta\left(\frac{1}{q}-\frac{m}{n}\right)
	+
	\frac{1-\theta}{r},
	\qquad
	\theta\in\left[\frac{j}{m},1\right].
	\]
	Then there exists a constant $c=c(j,m,n,p,q,r)>0$ such that
	\[
	\|\partial_x^{\alpha} f\|_{L^p(\mathbb{R}^n)}
	\le
	c
	\sum_{|\beta|=m}
	\|\partial^{\beta} f\|_{L^q(\mathbb{R}^n)}^{1-\theta}
	\|f\|_{L^r(\mathbb{R}^n)}^{\theta}.
	\]
\end{lemma}

\begin{lemma}\label{auxlemma4.1}
	Let $f \in H^1(\mathbb{R})$ and let $\psi$ be a smooth non-negative function. Then, for every $\varepsilon > 0$, the following inequality holds:
	\begin{align*}
		\int_{\mathbb{R}} f^3(x) \psi(x)\, dx \leq\;&
		\varepsilon \int_{\mathbb{R}} (f'(x))^2 \psi(x)\, dx
		+ \frac{3}{4}\varepsilon^{-1/3}
		\left(\int_{\mathbb{R}}\psi(x)f^2(x)\, dx\right)
		\left(\int_{\mathbb{R}}f^2(x)\, dx\right)^{2/3}\\
		&+
		\left(\int_{\mathbb{R}}|\psi'(x)|f^2(x)\, dx\right)^{1/2}
		\left(\int_{\mathbb{R}}\psi(x)f^2(x)\, dx\right)^{1/2}
		\left(\int_{\mathbb{R}}f^2(x)\, dx\right)^{1/2}.
	\end{align*}
\end{lemma}

\begin{proof}
	See inequality (1.18) in \cite{CavFamNat}.
\end{proof}

We now employ the Gagliardo--Nirenberg inequality together with Lemma \ref{auxlemma4.1} to establish the following estimate, which will be crucial for our analysis.

\begin{lemma}\label{VRnL2.3}
	Let $f \in H^1(\mathbb{R}^n)$ with $n \in \{2,3\}$, and let $\psi$ be a smooth non-negative function in $W^{1,\infty}(\mathbb{R})$ satisfying $0 \leq \psi' \leq \psi$. Then, for every $\varepsilon > 0$, there exists a constant $c(n,\varepsilon,\|\psi\|_{L^\infty})>0$ such that
	\begin{equation}\label{ineq1.1}
		\left|\int_{\mathbb{R}^n} \psi(x_1)f^3(x)\, dx\right|
		\le
		\varepsilon \int_{\mathbb{R}^n} \psi(x_1)|\nabla f(x)|^2\, dx
		+
		c(n,\varepsilon,\|\psi\|_{L^\infty})
		\left(
		\|f\|_{L^2(\mathbb{R}^n)}^{\frac{2(6-n)}{4-n}}
		+
		\|f\|_{L^2(\mathbb{R}^n)}^{2}
		\right).
	\end{equation}
\end{lemma}

\begin{proof}
	By Lemma \ref{auxlemma4.1} and the hypothesis $0 \leq \psi' \leq \psi$, we have
	\begin{align}\label{ineq6.1}
		\left|\int_{\mathbb{R}^n} \psi(x_1)f^3(x) \, dx\right|
		\leq\;&
		\varepsilon\int_{\mathbb{R}^{n - 1}}\left[\int_{\mathbb{R}}\psi(x_1) (\partial_{x_1}f(x_1,y))^2 \, dx_1\right] dy\nonumber\\
		&+ \int_{\mathbb{R}^{n - 1}} \left[ c(\varepsilon)\left(\int_{\mathbb{R}} \psi(x_1) f^2(x_1,y)\, dx_1\right)\left(\int_{\mathbb{R}} f^2(x_1,y)\, dx_1\right)^{\frac{2}{3}}\right] dy\nonumber \\
		&+ \int_{\mathbb{R}^{n - 1}} \left[ \left(\int_{\mathbb{R}} \psi(x_1) f^2(x_1,y)\, dx_1\right)\left(\int_{\mathbb{R}} f^2(x_1,y)\, dx_1\right)^{\frac{1}{2}}\right] dy\nonumber\\
		& =: \varepsilon\int_{\mathbb{R}^n}\psi(x_1) (\partial_{x_1}f)^2 \, dx + ( \ast_1) + (\ast_2).
	\end{align}
	
	To estimate $(\ast_1)$ and $(\ast_2)$ we combine Hölder's inequality, Minkowski's inequality for integrals, and the Gagliardo--Nirenberg inequality to obtain
	\begin{align}\label{ineq5.1}
		( \ast_1)
		&\leq c(\varepsilon)\left(\int_{\mathbb{R}}  \psi(x_1) \| f(x_1, \cdot) \|_{L^6(\mathbb{R}^{n - 1})}^2 \, dx_1\right) \cdot \|f\|_{L^2(\mathbb{R}^n)}^{\frac{4}{3}}\nonumber\\
		& = c(\varepsilon)\left(\int_{\mathbb{R}} \psi(x_1)^{\frac{n-1}{3}}  \|\nabla_y f(x_1, \cdot) \|_{L^2(\mathbb{R}^{n - 1})}^{\frac{2(n - 1)}{3}}\psi(x_1)^{\frac{4-n}{3}} \|f(x_1, \cdot)\|_{L^2(\mathbb{R}^{n - 1})}^{\frac{2(4 - n)}{3}} \, dx_1\right) \cdot \|f\|_{L^2(\mathbb{R}^n)}^{\frac{4}{3}}\nonumber\\
		& \leq c(\varepsilon)\left(\int_{\mathbb{R}}  \psi(x_1) \|\nabla_y f(x_1, \cdot) \|_{L^2(\mathbb{R}^{n - 1})}^2 \, dx_1\right)^{\frac{n - 1}{3}}\nonumber\\ 
		&\hspace{1.0cm}\times\left(\int_{\mathbb{R}}  \psi(x_1) \|f(x_1, \cdot)\|_{L^2(\mathbb{R}^{n - 1})}^2 \, dx_1\right)^{\frac{4 - n}{3}}\|f\|_{L^2(\mathbb{R}^n)}^{\frac{4}{3}}\nonumber\\
		& \leq c(\varepsilon)\| \psi \|_{L^{\infty}(\mathbb{R})}^{\frac{4-n}{3}}\left(\int_{\mathbb{R}^n}  \psi(x_1)|\nabla f(x)|^2 \, dx_1 \right)^{\frac{n - 1}{3}} \|f\|_{L^2(\mathbb{R}^n)}^{\frac{2(6 - n)}{3}}\nonumber\\
		&\leq \varepsilon \int_{\mathbb{R}^n}  \psi(x_1) |\nabla f(x)|^2 \, dx + c( n, \varepsilon, \| \psi \|_{L^{\infty}})\|f\|_{L^2(\mathbb{R}^n)}^{\frac{2(6 - n)}{4-n}},
	\end{align}
	
	\begin{align}\label{ineq7.1}
		( \ast_2)
		&\leq \left\|\int_{\mathbb{R}} \psi(x_1) f^2(x_1,\cdot)\, dx_1\right\|_{L^2(\mathbb{R}^{n - 1})} \| f \|_{L^2(\mathbb{R}^n)}\nonumber\\
		& \leq \left(\int_{\mathbb{R}} \psi(x_1) \|f(x_1, \cdot) \|_{L^4(\mathbb{R}^{n - 1})}^2 \, dx_1\right) \| f \|_{L^2(\mathbb{R}^n)}\nonumber\\
		& \leq \left(\int_{\mathbb{R}} \psi(x_1) \|\nabla f(x_1, \cdot) \|_{L^2(\mathbb{R}^{n - 1})}^{\frac{n - 1}{2}} \| f(x_1, \cdot) \|_{L^2(\mathbb{R}^{n - 1})}^{\frac{5 - n}{2}} \, dx_1\right)\| f \|_{L^2(\mathbb{R}^n)}\nonumber\\
		& \leq \left(\int_{\mathbb{R}} \psi(x_1) \|\nabla f(x_1, \cdot) \|_{L^2(\mathbb{R}^{n - 1})}^2 \, dx_1\right)^{\frac{n-1}{4}} 
		\left(\int_{\mathbb{R}} \psi(x_1) \|f(x_1, \cdot) \|_{L^2(\mathbb{R}^{n - 1})}^2 \, dx_1\right)^{\frac{5 - n}{4}}\| f \|_{L^2(\mathbb{R}^n)}\nonumber\\
		& \leq \| \psi\|_{L^{\infty}(\R)}^{\frac{5 - n}{4}}\left(\int_{\mathbb{R}^n} \psi(x_1) |\nabla f(x) |^2 \, dx_1\right)^{\frac{n - 1}{4}}  \| f \|_{L^2(\mathbb{R}^n)}^{\frac{5-n}{2}}\nonumber\\
		& \leq \varepsilon \int_{\mathbb{R}^n}  \psi(x_1) |\nabla f(x)|^2 \, dx + c( n, \varepsilon, \| \psi \|_{L^{\infty}})\| f \|_{L^2(\mathbb{R}^n)}^{2}.
	\end{align}
	
	Gathering \eqref{ineq6.1}, \eqref{ineq5.1}, and \eqref{ineq7.1}, we obtain \eqref{ineq1.1}.
\end{proof}

\section{Well-posedness and the observability inequality} 

In what follows, we present the well-posedness result needed in this work. The details of the proof of local well-posedness will be omitted. For $n=2$, the result follows from the methods developed by Faminskii \cite{Fa} and by Linares and Pastor \cite{LPas}. For $n=3$, the proof is a direct adaptation of a more comprehensive result established by Herr and Kinoshita in \cite{HK}. These results essentially lead to the following well-posedness result for the IVP with damping \eqref{eqZK1dampped}.

\begin{theorem}\label{TWP}
	Consider the IVP \eqref{eqZK1dampped} with $u_0\in H^1(\R^n)$ and $a\in W^{1,\infty}(\R^n)$ satisfying \eqref{dc2}. Given any $T>0$, there exists a space $X_T\subseteq C([0,T];H^1(\mathbb R^n))$ such that \eqref{eqZK1dampped} admits a unique mild solution $u\in X_T$ defined on the interval $[0,T]$, i.e.,
	\begin{equation}
		u(\cdot,t)= U(t)u_0+\dis\int_0^tU(t-\tau)(u\p_{x_1}u+au)(\tau)\,d\tau,
	\end{equation}
	and
	\begin{equation}
		u\p_{x_1}u,\, a(x)u\in L^2([0,T];L^2(\R^n)),
	\end{equation}
	where $\{U(t)\}_{t\in\R}$ is the unitary group defined by
	\[
	U(t)\phi(x)=\dis\int_{\R^n}e^{i(x\cdot\xi+t\xi_1|\xi|^2)}\widehat{\phi}(\xi)\, d\xi.
	\]
	Moreover, for any $T'\in(0,T)$ there exists a neighborhood $V$ of $u_0$ such that the map $u_0\longmapsto u(t)$ from $H^1(\R^n)$ into $X_T$ is Lipschitz continuous.
\end{theorem}

\textbf{Proof of local well-posedness.}
In order to establish local well-posedness, we treat the cases $n=2$ and $n=3$ separately. 

For $n=2$, it is enough to use the method of proof of Theorem 1.1 in \cite{Fa} or Theorem 1.6 in \cite{LPas}, with
\begin{equation} \label{Z}
	X_{T}=\{u \in C([0,T]; H^1(\mathbb R^2)) \ : \ \|u\|_{X_{T}}<\infty\},
\end{equation}
where
\begin{displaymath}
	\|u\|_{X_{T}}:=\left\| u\right\| _{L^{\infty}_{T}H^1(\R^n)}+\|u\|_{L^{\infty}_{x}L^{2}_{T,y}}+\left\| u\right\|_{L^{3}_{T}L^{\infty}_{x,y}}+\|u\|_{L^{2}_xL^{\infty}_{T,y}}.
\end{displaymath}

The case $n=3$ follows by using the method of proof of Theorem 1.1 in \cite{HK}. In this case,
\begin{equation} \label{Z2}
	X_{T}:=X^{1,b}\subseteq C([0, T];H^1(\mathbb R^3)),
\end{equation}
where $X^{s,b}$, $s,b\in\R$, is defined in \cite{HK}, Section 2.1, p. 1207.

\textbf{Proof of global well-posedness.}
In order to extend the local solution globally in time, we use the following two quantities presented in the next lemma.

\begin{lemma}\label{decayL2lemma}
	Let $u$ be the local mild solution of \eqref{eqZK1dampped} with initial data $u_0 \in H^1(\mathbb{R}^n)$. Then, for any $t\in(0,T]$ one has
	\begin{equation}\label{eq1al}
		\int_{\mathbb{R}^n} u^2(x, t)\, dx + 2\int_0^t\int_{\mathbb{R}^n} a(x)u^2(x, \tau)\, dx\, d\tau = \int_{\mathbb{R}^n}u_0(x)^2\, dx
	\end{equation}
	and
	\begin{align}\label{eq2al}
		\int_{\mathbb{R}^n}&\left(|\nabla u(x, t)|^2-\frac13u^3(x, t)\right)dx
		+\int_0^t\int_{\mathbb{R}^n}\left\{\nabla a(x)\cdot\nabla (u^2(x, \tau))+2a(x)|\nabla u(x, t)|^2\right\}dx\, d\tau\nonumber\\ 
		&=\int_{\mathbb{R}^n} \left(|\nabla u_0|^2-\frac13u_0^3\right)dx+\int_0^t\int_{\mathbb{R}^n}a(x)u^3(x, \tau)dx\, d\tau.
	\end{align}	 
\end{lemma}

\begin{proof}
	Let $u$ be a smooth solution of \eqref{eqZK1dampped}. Multiplying equation \eqref{eqZK1dampped} by $2u$ and integrating with respect to $x$, we obtain
	\begin{equation}\label{eq1al1}
		\dfrac{d}{dt}\dis\int_{\R^n}u^2(x,t)\,dx
		+2\dis\int_{\R^n}u(x,t)\Delta\p_{x_1}u(x,t)\,dx
		+2\dis\int_{\R^n}a(x)u^2(x,t)\,dx=0.
	\end{equation}
	Since
	\[
	2\dis\int_{\R^n}u(x,t)\Delta\p_{x_1}u(x,t)\,dx=0,
	\]
	integrating \eqref{eq1al1} from $0$ to $t$ we obtain \eqref{eq1al}.
	
	Finally, multiplying equation \eqref{eqZK1dampped} by $-(2\Delta u+u^2)$ and integrating with respect to $x$, we obtain
	\begin{align}\label{eq2al1}
		\dfrac{d}{dt}\int_{\mathbb{R}^n}\hspace{-0.2cm}\left(|\nabla u|^2-\frac13u^3\right)dx
		+\int_{\mathbb{R}^n}\hspace{-0.2cm}\nabla a(x)\cdot\nabla (u^2)dx
		+2\int_{\mathbb{R}^n}\hspace{-0.3cm}a(x)|\nabla u|^2dx
		-\int_{\mathbb{R}^n}\hspace{-0.3cm}a(x)u^3dx=0.
	\end{align}
	Thus, integrating \eqref{eq2al1} from $0$ to $t$ we obtain \eqref{eq2al}.
\end{proof}

Notice that conditions \eqref{damping cond} and \eqref{eq1al} imply
\begin{equation*}
	\dfrac{d}{dt}\dis\int_{\R^n}u^2(x,t)\,dx\leq-2\alpha_0\dis\int_{\R^n}u^2(x,t)\,dx,
\end{equation*}
which leads to
\begin{equation}\label{ineq1al}
	\|u(t)\|^2_{L^2(\R^n)} \leq e^{-2\alpha_0t}\|u_0\|^2_{L^2(\R^n)}.
\end{equation}

With equalities \eqref{eq1al} and \eqref{eq2al}, we are ready to prove that the IVP \eqref{eqZK1dampped} is globally well-posed in the class $X_T$, for initial data $u_0\in H^1(\R^n)$ and a non-negative function $a(x)$. From \eqref{eq1al}, we obtain
\begin{equation}\label{L2}
	\|u(t)\|_{L^2(\R^n)}\leq\|u_0\|_{L^2(\R^n)}.
\end{equation}

Employing in \eqref{eq2al1} the same tools used to obtain \eqref{eq10theo2.1}, together with \eqref{L2}, we obtain
\begin{equation}\label{D2}
	\dfrac{d}{dt}\|\nabla u(t)\|^2_{L^2(\R^n)}+ 2(\alpha_0 - \varepsilon)\|\nabla u(t)\|^2_{L^2(\R^n)} \leq c(\varepsilon, \| u_0 \|_{L^2}, \| a \|_{W^{1, \infty}}),
\end{equation}
and therefore
\begin{equation}\label{D3}
	\|\nabla u(t)\|^2_{L^2(\R^n)}\leq \|\nabla u_0\|^2_{L^2(\R^n)}+c(\varepsilon, \alpha_0, \| u_0 \|_{L^2}, \| a \|_{W^{1, \infty}}).
\end{equation}
Estimates \eqref{L2} and \eqref{D3} allow us to extend the local solution $u$ to any time interval.


\begin{lemma}\label{lemma1.1aux}
	Consider an integer $n\in\{2,3\}$ and let $u \in C([0, T]; H^1(\mathbb{R}^n))$ be a mild solution of \eqref{eqZK1dampped} with initial data $u_0 \in H^1(\mathbb{R}^n)$. Let $\psi$ be a smooth non-negative function in $W^{1,\infty}(\mathbb{R})$ such that $0 \leq \psi' \leq \psi$. Then, for all $\varepsilon > 0$ there exists a constant $c = c(\varepsilon, \| u_0\|_{L^2}, \| \psi \|_{W^{1,\infty}}) > 0$ such that
	\begin{equation}\label{eq1.1auxlemmas}
		\left|\int_{\mathbb{R}^n} \psi(x_1)u^3(x,t)\,dx\right|
		\leq
		\varepsilon\int_{\mathbb{R}^n} \psi(x_1)|\nabla u(x,t)|^2\,dx + c.
	\end{equation}
\end{lemma}

\begin{proof}
	By using Lemma \ref{VRnL2.3} with $f=u$ we have, for each $t\in(0,T]$, that
	\begin{align}\label{ineq6.1aux}
		\left|\int_{\mathbb{R}^n} \psi(x_1)u^3\,dx\right|
		\leq\;&
		\varepsilon\int_{\mathbb{R}^n} \psi(x_1)|\nabla u(x,t)|^2\,dx\nonumber\\ 
		&+
		c( n, \varepsilon, \| \psi \|_{L^{\infty}})
		\bigg{(}
		\|u(\cdot,t)\|_{L^2(\mathbb{R}^n)}^{\frac{2(6 - n)}{4-n}}
		+
		\|u(\cdot,t)\|_{L^2(\mathbb{R}^n)}^{2}
		\bigg{)}.
	\end{align}
	
	Thus we obtain \eqref{eq1.1auxlemmas} from \eqref{ineq6.1aux} together with inequality \eqref{ineq1al}.
\end{proof}

\begin{lemma}[Smoothing effect]\label{lemma1aux}
	Consider $n\in\{2,3\}$. Let $u \in C([0, T]; H^1(\mathbb{R}^n))$ be a mild solution of \eqref{eqZK1dampped}. Then, for any $r>0$ there exists a constant $c = c(T, r, \|u_0\|_{L^2(\mathbb{R}^n)})$ such that
	\begin{equation}\label{eq0auxlemmas}
		\int_0^T\int_{Q_r} |\nabla u(x, t)|^2\,dx\,dt \leq c,
	\end{equation}
where $Q_r=[-r, r] \times \mathbb{R}^{n - 1}$.	
\end{lemma}

\begin{proof}
	We adapt the argument presented in \cite{CavFamNat} to the $n$-dimensional setting. The idea was originally introduced by Kato in \cite{Ka}. Let $\rho \in C^{\infty}(\mathbb{R})$ be a smooth, bounded, positive, and increasing function with bounded derivatives such that $\rho'(x_1) \equiv 1$ for $x_1 \in [-r, r]$ and $|\rho''(x_1)| \leq \rho'(x_1)$. Notice that $\rho$ and its derivatives are bounded by a constant depending on $r$. 
	
	Assume that $u$ is sufficiently smooth. We multiply equation \eqref{eqZK1dampped} by $u(x,t)\rho(x_1)$ and integrate over $x \in \mathbb{R}^n$. Using integration by parts and then integrating the resulting expression with respect to $t\in(0,T]$, we obtain
\begin{align}\label{eq3.0katosmooth}
	&\dfrac{1}{2}\int_{\mathbb{R}^n}u(x,t)^2\rho(x_1)\,dx
	+ \dfrac{1}{2}\int_0^t\int_{\mathbb{R}^n}[2u_{x_1}^2(x,\tau) + |\nabla u(x,\tau)|^2]\rho'(x_1)\,dx\,d\tau\nonumber\\
	&- \dfrac{1}{2}\int_0^t\int_{\mathbb{R}^n}u^2(x,\tau)\rho'''(x_1)\,dx\,d\tau
	-\dfrac{1}{3}\int_0^t\int_{\mathbb{R}^n}u^3(x,\tau)\rho'(x_1)\,dx\,d\tau\nonumber\\ 
	&+ \int_0^t\int_{\mathbb{R}^n}a(x)u^2(x,\tau)\rho(x_1)\,dx\,d\tau
	= \dfrac{1}{2}\int_{\mathbb{R}^n}u_0(x)^2\rho(x_1)\,dx.
\end{align}

Since $a(x)$ and $\rho(x_1)$ are non-negative functions, it follows from \eqref{eq3.0katosmooth} that
\begin{align}\label{eq3katosmooth} 
	\int_0^t\!\!\!\int_{\mathbb{R}^{n}} |\nabla u(x,\tau)|^2 \rho'(x_1)\,dx\,d\tau
	\leq\;&  \int_{\mathbb{R}^n}u_0(x)^2\rho(x_1)\,dx
	+ \int_0^t\!\!\!\int_{\mathbb{R}^n}u^2(x,\tau)\rho'''(x_1)\,dx\,d\tau\\ 
	&+\dfrac{2}{3}\int_0^t\!\!\!\int_{\mathbb{R}^n}u^3(x,\tau)\rho'(x_1)\,dx\,d\tau.\nonumber
\end{align}
Therefore, using the boundedness of $\rho$ and its derivatives and applying Lemma \ref{lemma1.1aux} to the last term in \eqref{eq3katosmooth}, we obtain
\begin{align*}
	\int_0^t\!\!\int_{\mathbb{R}^n}|\nabla u(x,\tau)|^2\rho'(x_1)\,dx\,d\tau
	\leq& c(r)\int_{\mathbb{R}^n}u_0(x)^2\,dx 
	+ c(r,\varepsilon)\int_0^t\!\!\int_{\mathbb{R}^n}u^2(x,\tau)\,dx\,d\tau\\ 
	&+\varepsilon\int_0^t\!\!\int_{\mathbb{R}^n}|\nabla u(x,\tau)|^2\rho'(x_1)\,dx\,d\tau.
\end{align*}
Hence, using that $\rho' \equiv 1$ on $[-r,r]$ and choosing $\varepsilon = \tfrac12$, we conclude the desired estimate.
\end{proof}
\begin{lemma}[Observability inequality]\label{lemma2aux}
	Suppose that $a(x)$ is a function satisfying the hypotheses of Theorem \eqref{teo1}. Then, for any $T,L>0$, there exists a constant $c(T,L)>0$ such that every mild solution $u \in C([0,T]; H^1(\mathbb{R}^n))$ of \eqref{eqZK1dampped} with initial data satisfying $\|u_0\|_{H^1(\mathbb{R}^n)} \leq L$ satisfies
	\begin{equation}\label{observability}
		\int_0^T \int_{Q_R} u^2(x,t)\,dx\,dt \leq c \int_0^T \int_{\mathbb{R}^n} a(x)u^2(x,t)\,dx\,dt.
	\end{equation}
\end{lemma}

\begin{proof}
Arguing by contradiction,	assume that \eqref{observability} is false. Then, for each $k \in \mathbb{N}$ there exists a sequence $\{u_k\}$ of mild solutions of \eqref{eqZK1dampped} corresponding to initial data $u_{k,0} \in H^1(\mathbb{R}^n)$ satisfying $\|u_{k,0}\|_{L^2(\mathbb{R}^n)} \leq L$ such that
	\begin{equation}\label{proofobserveq1}
		\int_0^T \int_{Q_R} u_k^2(x,t)\,dx\,dt \geq k \int_0^T \int_{\mathbb{R}^n} a(x)u_k^2(x,t)\,dx\,dt.
	\end{equation}
	Consequently,
	\begin{equation}\label{proofobserveq2}
		\lim_{k\to +\infty}
		\dfrac{\int_0^T \int_{\mathbb{R}^n} a(x)u_k^2(x,t)\,dx\,dt}
		{\int_0^T \int_{Q_R} u_k^2(x,t)\,dx\,dt}
		=0.
	\end{equation}
	
	We define
	\begin{equation}\label{proofobserveq3}
		\gamma_k := \|u_k\|_{L^2([0,T];L^2(Q_R))}, 
		\qquad 
		v_k := \frac{u_k}{\gamma_k},
	\end{equation}
	so that
	\begin{equation}\label{proofobserveq4}
		\|v_k\|_{L^2([0,T];L^2(Q_R))} = 1,
		\qquad \forall\, k \in \mathbb{N}.
	\end{equation}
	
	Notice that each $v_k$ is a mild solution of
	\begin{equation}\label{proofobserveq5}
		\left\{
		\begin{array}{l}
			\partial_t v_k + \partial_{x_1}\Delta v_k + \gamma_k v_k \partial_{x_1}v_k + a(x)v_k = 0,\\
			v_k(0,x) = v_{k,0},
		\end{array}
		\right.
	\end{equation}
	where
	\[
	v_{k,0} = \frac{u_{k,0}}{\gamma_k}.
	\]
	
	We intend to prove that a subsequence of $v_k$ converges to a function $v$ which is a solution of the Zakharov--Kuznetsov equation. To this end we first obtain suitable bounds for the sequence $\{v_k\}_{k\in\mathbb{N}}$.
	
	Note that Lemma \ref{lemma1aux} still holds for solutions of \eqref{proofobserveq5}, with constants depending on $\gamma_k$ and $\|v_{k,0}\|_{L^2}$. More precisely, 
	\begin{equation}\label{proofobserveq13}
		\int_0^T \int_{[-r,r]\times\mathbb{R}^{n-1}} |\nabla v_k(x,t)|^2\,dx\,dt
		\leq c(r,T,\gamma_k,\|v_{k,0}\|_{L^2}), \ \ \forall \,r>0.
	\end{equation}
Therefore, it remains to prove that $\{\gamma_k\}$ and $\{\|v_{k,0}\|_{L^2}\}$ are bounded sequences. Indeed, since
\[
\int_0^T \int_{\mathbb{R}^n} u_k^2(x,t)\,dx\,dt \leq T \| u_{k,0}\|_{L^2(\mathbb{R}^n)}^2,
\]
we have
\begin{equation}\label{proofobserveq6}
	\gamma_k \leq T^{1/2}L, \qquad \forall\, k \in \mathbb{N}.
\end{equation}

Next we prove that $\{v_{k,0}\}$ is bounded in $L^2(\mathbb{R}^n)$. Since $a(x) \geq \alpha_0$ for all $x \notin Q_R$, we have
\begin{align}\label{proofobserveq9}
	\int_0^T \int_{\mathbb{R}^n} u_k^2(x,t)\,dx\,dt
	&\leq \int_0^T \int_{Q_R} u_k^2(x,t)\,dx\,dt
	+ \dfrac{1}{\alpha_0}\int_0^T \int_{Q_R^c} a(x) u_k^2(x,t)\,dx\,dt \nonumber\\
	&\leq \int_0^T \int_{Q_R} u_k^2(x,t)\,dx\,dt
	+ \dfrac{1}{\alpha_0}\int_0^T \int_{\mathbb{R}^n} a(x) u_k^2(x,t)\,dx\,dt.
\end{align}

On the other hand, integrating \eqref{eq1al} over $[0,T]$ yields
\begin{equation}\label{proofobserveq10}
	\int_0^T \int_{\mathbb{R}^n} u_k^2(x,t)\,dx\,dt
	+ 2T \int_0^T \int_{\mathbb{R}^n} a(x)u_k^2(x,\tau)\,dx\,d\tau
	\geq
	T \int_{\mathbb{R}^n} u_{k,0}(x)^2\,dx.
\end{equation}

Thus, from \eqref{proofobserveq9} and \eqref{proofobserveq10} it follows that
\begin{align}\label{proofobserveq11}
	\int_{\mathbb{R}^n} u_{k,0}(x)^2\,dx
	&\leq \dfrac{1}{T} \int_0^T \int_{\mathbb{R}^n} u_k^2(x,t)\,dx\,dt
	+ 2 \int_0^T \int_{\mathbb{R}^n} a(x)u_k^2(x,t)\,dx\,dt \nonumber\\
	&\leq \dfrac{1}{T} \int_0^T \int_{Q_R} u_k^2(x,t)\,dx\,dt
	+ \left(2 + \dfrac{1}{\alpha_0 T}\right)
	\int_0^T \int_{\mathbb{R}^n} a(x) u_k^2(x,t)\,dx\,dt \nonumber\\
	&\leq \dfrac{\gamma_k^2}{T}
	+ \gamma_k^2 \left(2 + \dfrac{1}{\alpha_0 T}\right)
	\int_0^T \int_{\mathbb{R}^n} a(x) v_k^2(x,t)\,dx\,dt.
\end{align}

Dividing both sides of \eqref{proofobserveq11} by $\gamma_k^2$ we obtain
\begin{equation}\label{proofobserveq12}
	\int_{\mathbb{R}^n} v_{k,0}(x)^2\,dx
	\leq
	\dfrac{1}{T}
	+
	\left(2 + \dfrac{1}{\alpha_0 T}\right)
	\int_0^T \int_{\mathbb{R}^n} a(x) v_k^2(x,t)\,dx\,dt.
\end{equation}

Finally, note that from \eqref{proofobserveq2} we have
\begin{equation}\label{proofobserveq8}
	\lim_{k \to +\infty}
	\int_0^T \int_{\mathbb{R}^n} a(x) v_k^2(x,t)\,dx\,dt = 0.
\end{equation}

Combining \eqref{proofobserveq8} and \eqref{proofobserveq12}, we conclude that
\begin{equation}\label{proofobserveq14}
	\sup_{k \in \mathbb{N}} \| v_{k,0} \|_{L^2(\mathbb{R}^n)} < \infty.
\end{equation}
Gathering \eqref{proofobserveq13}, \eqref{proofobserveq6}, and \eqref{proofobserveq14}, we obtain that
\begin{equation}\label{proofobserveq15}
	\{v_k\}_{k \in \mathbb{N}} \ \text{is bounded in}\ 
	L^{\infty}([0,T];L^2(\mathbb{R}^n)) \cap L^2([0,T];H^1(Q_r)), 
	\quad \forall\, r>0.
\end{equation}

Therefore, up to a subsequence, we have that there exist $\gamma \in \mathbb{R}$, $v_0 \in L^2(Q_r)$, and 
$v \in L^{\infty}([0,T];L^2(\mathbb{R}^n))\, \cap L^2([0,T];H^1(Q_r))$ such that 
\begin{align}\label{proofobserveq23}
	\gamma_k &\longrightarrow \gamma, \nonumber\\
	v_{k,0} &\rightharpoonup v_0 \quad \text{in } L^2(Q_r),\\
	v_k &\rightharpoonup v \quad \text{in } L^2([0,T];H^1(Q_r)), \nonumber\\
	v_k &\stackrel{\ast}{\rightharpoonup} v 
	\quad \text{in } L^{\infty}([0,T];L^2(Q_r)). \nonumber
\end{align}

Next we prove
\begin{equation}\label{proofobserveq21}
	v_k \longrightarrow v \quad \text{in } L^2([0,T]\times\mathbb{R}^n).
\end{equation}

To this end, we use integration by parts and the Gagliardo--Nirenberg inequality
\[
\| v_k \|_{L^4(\mathbb{R}^n)} 
\leq 
c \| v_k \|_{L^2(\mathbb{R}^n)}^{\frac{n}{4}}
\|\nabla v_k \|_{L^2(\mathbb{R}^n)}^{\frac{4-n}{4}},
\qquad (2 \leq n \leq 4),
\]
to conclude that for every smooth compactly supported function $\varphi$ in $Q_r \times (0,T)$, 
\begin{align*}
	\left| \int_0^T\int_{Q_r} (v_k \partial_{x_1}v_k)(x,t)\varphi(x,t)\,dx\,dt \right|
	&= \left| \int_0^T\int_{Q_r} v_k^2(x,t)\partial_{x_1}\varphi(x,t)\,dx\,dt \right|\\
	\leq& \int_0^T
	\|v_k(\cdot,t)\|_{L^4(Q_r)}^2
	\|\varphi(\cdot,t)\|_{H^1(Q_r)}\,dt\\
	\leq & \int_0^T
	\|v_k(\cdot,t)\|_{L^2(\mathbb{R}^n)}^{\frac{n}{2}}
	\|v_k(\cdot,t)\|_{H^1(Q_r)}^{\frac{4-n}{2}}
	\|\varphi(\cdot,t)\|_{H^1(Q_r)}\,dt\\
	\leq &
	\|v_k\|_{L^\infty(0,T;L^2(\mathbb{R}^n))}^{\frac{n}{2}}
	\int_0^T
	\|v_k(\cdot,t)\|_{H^1(Q_r)}^{\frac{4-n}{2}}
	\|\varphi(\cdot,t)\|_{H^1(Q_r)}\,dt\\
	\leq &
	T^{\frac{n}{4}}
	\|v_k\|_{L^\infty(0,T;L^2(\mathbb{R}^n))}^{\frac{n}{2}}
	\|v_k\|_{L^2(0,T;H^1(Q_r))}^{\frac{4-n}{2}}
	\|\varphi\|_{L^2(0,T;H^1(Q_r))}.
\end{align*}

Therefore
\begin{equation}\label{proofobserveq20}
	\{v_k\partial_{x_1}v_k\}_{k \in \mathbb{N}}
	\ \text{is bounded in}\ 
	L^2([0,T];H^{-1}(Q_r)),
	\qquad \forall r>0.
\end{equation}
Consequently\footnote{If we only proved that the sequence $\{v_k\partial_{x_1}v_k\}_{k\in\mathbb{N}}$ is bounded in $L^2([0,T];H^{-2}(Q_r))$, we would still obtain \eqref{proofobserveq16}. This would allow one to enlarge the admissible range of dimensions $n$ slightly. However, since we are mainly concerned with the cases $n=2,3$, we do not pursue these additional technicalities.}, since $v_k$ satisfies \eqref{proofobserveq5}, we have
\begin{equation}\label{proofobserveq16}
	\{\partial_t v_k\}_{k\in\mathbb{N}}
	\ \text{is bounded in}\ 
	L^2([0,T];H^{-2}(Q_r)),
	\qquad \forall r>0.
\end{equation}

Hence $\{v_k\}$ is a bounded sequence in
\[
W_r := \left\{ u \in L^2([0,T];H^1(Q_r)) \ : \ \partial_t u \in L^2([0,T];H^{-2}(Q_r)) \right\}.
\]

Since $H^1(Q_r)$ is compactly embedded in $L^2(Q_r)$ and $L^2(Q_r)$ is continuously embedded in $H^{-2}(Q_r)$, it follows from the Aubin--Lions theorem that
\begin{equation}\label{proofobserveq17}
	v_k \longrightarrow v
	\quad \text{in } L^2([0,T];L^2(Q_r)),
	\qquad \forall r>0.
\end{equation}

On the other hand, from \eqref{proofobserveq8} and the hypothesis on the damping function $a(x)$, we have
\[
v_k \longrightarrow 0
\quad \text{in } L^2([0,T]\times Q_R^{c}).
\]

This, together with \eqref{proofobserveq17}, allows us to conclude that
\begin{equation}\label{proofobserveq18}
	v_k \longrightarrow v
	\quad \text{in } L^2([0,T]\times\mathbb{R}^n),
\end{equation}
where $v$ satisfies
\begin{equation}\label{proofobserveq19}
	v(x,t)=0,
	\qquad \text{for all } x\in Q_R^{c}.
\end{equation}

Next, we claim that $v$ is a mild solution of
\begin{equation}\label{weak solution obsproof}
	\partial_t v + \partial_{x_1}\Delta v + \gamma v \partial_{x_1} v = 0.
\end{equation}

To this end, consider an arbitrary function $\phi(x,t)$ satisfying
\[
\phi \in C([0,T];H^3(\mathbb{R}^n)) 
\cap 
C^1([0,T];L^2(\mathbb{R}^n)),
\qquad
\phi(\cdot,T)\equiv 0.
\]

Multiplying equation \eqref{proofobserveq5} by $\phi$ and integrating by parts, we obtain
\begin{equation}\label{proofobserveq22}
	\int_0^T\int_{\mathbb{R}^n}
	\left(
	v_k \partial_t\phi
	+
	v_k\partial_{x_1}\Delta\phi
	+
	\gamma_k v_k^2 \partial_{x_1}\phi
	+
	a(x)v_k\phi
	\right)
	\,dx\,dt
	+
	\int_{\mathbb{R}^n}
	v_{k,0}\phi(x,0)\,dx
	=
	0.
\end{equation}

Letting $k\to+\infty$ in \eqref{proofobserveq22} and using the convergences in \eqref{proofobserveq23}, we conclude that
\begin{equation}\label{proofobserveq24}
	\int_0^T\int_{\mathbb{R}^n} 
	\left(
	v \partial_t \phi 
	+ v\partial_{x_1}\Delta \phi 
	+ \gamma v^2 \partial_{x_1}\phi  
	+ a(x)v \phi
	\right)
	\,dx\,dt 
	+ \int_{\mathbb{R}^n} v_0 \phi(x,0)\,dx = 0. 
\end{equation}
This means that $v$ is a weak solution of \eqref{weak solution obsproof}, as claimed.

Taking into account the smoothing effect for the solutions (see Theorem \ref{smooth solutions}), we deduce that $v$ is in fact a smooth solution of
\[
\partial_t v + \partial_{x_1}\Delta v + \gamma v \partial_{x_1}v = 0,
\]
which is compactly supported. Therefore, the unique continuation principle (see Theorem \ref{Tucp}) implies that $v \equiv 0$. However, this is a contradiction, since
\[
\| v \|_{L^2([0,T];L^2(\mathbb{R}^n))}
=
\lim_{k\to+\infty}
\| v_k \|_{L^2([0,T];L^2(Q_R))}
=
1.
\]
\end{proof}
\section{Proof of the main results}

\subsection{Proof of Theorem \ref{teo1}} Consider $T>0$. By the hypothesis on $a(x)$ and from \eqref{eq1al}, we have
\begin{align}\label{E2.3}
	\int_0^T \int_{|x_1|> R} u^2(x,t)\,dx\,dt
	\le
	\dfrac{1}{2\alpha_0}\int_{\mathbb{R}^n} u_0^2(x)\,dx
	-
	\dfrac{1}{2\alpha_0}\int_{\mathbb{R}^n} u^2(x,T)\,dx .
\end{align}

Then, from \eqref{E2.3}, the observability inequality \eqref{observability}, and the identity \eqref{eq1al}, we deduce that
\begin{align}\label{E2.4}
	\int_0^T\int_{\mathbb{R}^n} u^2(x,t)\,dx\,dt
	&\le
	\dfrac{1}{2\alpha_0}\int_{\mathbb{R}^n} u_0^2(x)\,dx
	-
	\dfrac{1}{2\alpha_0}\int_{\mathbb{R}^n} u^2(x,T)\,dx
	+
	\int_0^T\int_{Q_R} u^2(x,t)\,dx\,dt
	\nonumber\\
	&\hspace{-0.9cm}\le
	\dfrac{1}{2\alpha_0}\int_{\mathbb{R}^n} u_0^2(x)\,dx
	-
	\dfrac{1}{2\alpha_0}\int_{\mathbb{R}^n} u^2(x,T)\,dx+
	c(L)\int_0^T\int_{\mathbb{R}^n} a(x)u^2(x,t)\,dx\,dt
	\nonumber\\
	&\hspace{-0.9cm}=
	\left(\frac{1}{2\alpha_0}+\frac{c(L)}{2}\right)
	\int_{\mathbb{R}^n} u_0^2(x)\,dx
	-
	\left(\frac{1}{2\alpha_0}+\frac{c(L)}{2}\right)
	\int_{\mathbb{R}^n} u^2(x,T)\,dx .
\end{align}

Now, since
\[
T\int_{\mathbb{R}^n} u^2(x,T)\,dx
\le
\int_0^T\int_{\mathbb{R}^n} u^2(x,t)\,dx\,dt,
\]
it follows from \eqref{E2.4} that
\begin{align*}
	\left(T+\frac{1}{2\alpha_0}+\frac{c(L)}{2}\right)
	\int_{\mathbb{R}^n} u^2(x,T)\,dx
	\le
	\left(\frac{1}{2\alpha_0}+\frac{c(L)}{2}\right)
	\int_{\mathbb{R}^n} u_0^2(x)\,dx .
\end{align*}

Therefore,
\begin{align}
	\int_{\mathbb{R}^n} u^2(x,T)\,dx
	\le
	C(L,T)\int_{\mathbb{R}^n} u_0^2(x)\,dx,
\end{align}
where
\[
0<C(L,T):=
\frac{\frac{1}{2\alpha_0}+\frac{c(L)}{2}}
{T+\frac{1}{2\alpha_0}+\frac{c(L)}{2}}
<1 .
\]

Thus, in order to conclude the result, it suffices to apply the semigroup property.

\subsection{Proof of Theorem \ref{teo2}}

From identity \eqref{eq2al1} and the condition \eqref{dc2}, we see that
\[
H(t) := \int_{\mathbb{R}^n} \left(|\nabla u(x,t)|^2 - \frac{u^3(x,t)}{3}\right) dx
\]
satisfies
\begin{align}\label{eq1theo2.1}
	\hspace{-0.3cm}\frac{d}{dt}H(t) + 2\alpha_0 \int_{\mathbb{R}^n} |\nabla u|^2 dx
	&\le
	2\int_{\mathbb{R}^n} u \nabla a(x)\cdot\nabla u\, dx
	+ \int_{\mathbb{R}^n} a(x)u^3\,dx \nonumber\\
	&\le
	2\| \nabla a\|_{L^{\infty}(\mathbb{R}^n)}
	\|u\|_{L^2(\mathbb{R}^n)}
	\|\nabla u\|_{L^2(\mathbb{R}^n)}
	+
	\|a\|_{L^{\infty}(\mathbb{R}^n)}
	\|u\|_{L^3(\mathbb{R}^n)}^3 .
\end{align}

Using Young's inequality with parameter $\varepsilon$, we obtain
\begin{equation}\label{eq2theo2.1}
	2\| \nabla a\|_{L^{\infty}(\mathbb{R}^n)}
	\|u\|_{L^2(\mathbb{R}^n)}
	\|\nabla u\|_{L^2(\mathbb{R}^n)}
	\le
	\varepsilon\|\nabla u\|_{L^2(\mathbb{R}^n)}^2
	+
	\dfrac{1}{\varepsilon}
	\| \nabla a\|_{L^{\infty}(\mathbb{R}^n)}^2
	\|u\|_{L^2(\mathbb{R}^n)}^2 .
\end{equation}

Thus, from \eqref{eq1theo2.1} and \eqref{eq2theo2.1} we deduce
\begin{align}\label{eq4theo2.1}
	\frac{d}{dt}H(t)
	+
	(2\alpha_0-\varepsilon)
	\int_{\mathbb{R}^n} |\nabla u|^2 dx
	\le
	c(\varepsilon,\|a\|_{W^{1,\infty}})
	\left(
	\|u\|_{L^2(\mathbb{R}^n)}^2
	+
	\|u\|_{L^3(\mathbb{R}^n)}^3
	\right).
\end{align}

Hence,
\begin{align}\label{eq6theo2.1}
	\frac{d}{dt}H(t)
	+
	(2\alpha_0-\varepsilon)
	\int_{\mathbb{R}^n}
	\left(
	|\nabla u|^2
	-
	\frac{b}{3}u^3
	\right) dx
	&\le
	c(\varepsilon,\|a\|_{W^{1,\infty}})
	\|u\|_{L^2(\mathbb{R}^n)}^2 \nonumber\\
	&\quad +
	\left(
	\|a\|_{L^\infty(\mathbb{R}^n)}
	+
	\frac{(2\alpha_0-\varepsilon)b}{3}
	\right)
	\|u\|_{L^3(\mathbb{R}^n)}^3 ,
\end{align}
where $b\in\mathbb{R}$ will be chosen conveniently.

To control the $L^3$ norm we use the Gagliardo--Nirenberg inequality and Young's inequality as follows:
\begin{align}\label{eq5theo2.1}
	\|u\|_{L^3(\mathbb{R}^n)}^3
	&\lesssim
	\|\nabla u\|_{L^2(\mathbb{R}^n)}^{\frac{n}{2}}
	\|u\|_{L^2(\mathbb{R}^n)}^{\frac{6-n}{2}}
	\nonumber\\
	&\le
	\varepsilon
	\|\nabla u\|_{L^2(\mathbb{R}^n)}^2
	+
	c(\varepsilon)
	\|u\|_{L^2(\mathbb{R}^n)}^{\frac{2(6-n)}{4-n}},
	\qquad (n=2,3)
	\nonumber\\
	&\le
	\varepsilon
	\|\nabla u\|_{L^2(\mathbb{R}^n)}^2
	+
	c(\varepsilon,\|u_0\|_{L^2(\mathbb{R}^n)})
	\|u\|_{L^2(\mathbb{R}^n)}^2,
	\qquad \forall\, \varepsilon\ll1 .
\end{align}
Replacing \eqref{eq5theo2.1} into \eqref{eq6theo2.1}, we obtain
\begin{align}\label{eq7theo2.1}
	\frac{d}{dt}H(t)
	&+
	\left[
	2\alpha_0-\varepsilon
	-
	\varepsilon
	\left(
	\|a\|_{L^\infty(\mathbb{R}^n)}
	+
	\frac{2\alpha_0-\varepsilon}{3}b
	\right)
	\right]
	\int_{\mathbb{R}^n}|\nabla u|^2\,dx
	\nonumber\\
	&\quad
	-
	(2\alpha_0-\varepsilon)b
	\int_{\mathbb{R}^n}\frac{u^3}{3}\,dx
	\nonumber\\
	&\le
	c(b,\alpha_0,\varepsilon,\|a\|_{W^{1,\infty}},\|u_0\|_{L^2})
	\|u\|_{L^2(\mathbb{R}^n)}^2 .
\end{align}

We choose $b$ such that
\[
2\alpha_0-\varepsilon
-
\varepsilon
\left(
\|a\|_{L^\infty(\mathbb{R}^n)}
+
\frac{2\alpha_0-\varepsilon}{3}b
\right)
=
(2\alpha_0-\varepsilon)b ,
\]
namely
\begin{equation}\label{eq11theo2.1}
	b
	=
	3\cdot
	\frac{2\alpha_0-\varepsilon(1+\|a\|_{L^\infty})}
	{4(2\alpha_0-\varepsilon)} .
\end{equation}

Therefore, from \eqref{eq6theo2.1} and \eqref{eq5theo2.1} we obtain
\begin{equation}\label{eq12theo2.1}
	\frac{d}{dt}H(t)
	+
	(2\alpha_0-\varepsilon)b\,H(t)
	\le
	c(\varepsilon,\|u_0\|_{L^2},\|a\|_{W^{1,\infty}})
	\|u\|_{L^2(\mathbb{R}^n)}^2 .
\end{equation}

Since $\|u(t)\|_{L^2(\mathbb{R}^n)} \le C e^{-2\alpha_0 t}$, we conclude from \eqref{eq12theo2.1} that
\begin{equation}\label{eq10theo2.1}
	H(t)
	\le
	\tilde c(\varepsilon,\|u_0\|_{L^2},\|a\|_{W^{1,\infty}})
	e^{-(2\alpha_0-\varepsilon)bt}.
\end{equation}

Finally, using the Gagliardo--Nirenberg inequality as in \eqref{eq5theo2.1}, we obtain
\[
c\,\|\nabla u(t)\|_{L^2(\mathbb{R}^n)}^2
\le
H(t),
\]
for some positive constant $c$. Hence we conclude \eqref{H1dec}.


\section{Appendix}

Here we justify the regularity of the solutions to equation \eqref{eqZK1dampped}. This property is necessary in order to apply the unique continuation principle in the proof of Lemma \ref{lemma2aux}.

\begin{theorem}[Regularity]\label{smooth solutions}
	Let $m \geq 1$, and let $u \in C([0, T]; H^1(\mathbb{R}^n))$ be a mild solution of \eqref{eqZK1dampped} with support
	
	\hspace{3.0cm}$\supp u \subseteq Q_r:=[-r,r]\times\R^{n-1},$ \, for some $r>0$.\\\\
	If $u(t_0) \in H^m(\mathbb{R}^n)$ for some $t_0 \in (0, T)$, then $u(t) \in H^{m+1}(\mathbb{R}^n)$ for all $t \in [t_0, T]$.
\end{theorem}

\begin{proof}
	\textbf{Case} $m=2$. We consider $t_0=0$. First assume that $u$ is a smooth solution, so that it satisfies the differential equation. Let us denote $v_j=\partial_{x_j}u$. Differentiating \eqref{eqZK1dampped} with respect to the $j$-th coordinate, we obtain
	\begin{equation}\label{eq1smoothappendix}
		\partial_t v_j + \partial_{x_1}\Delta v_j + v_j v_1 + u\partial_{x_1}v_j + \partial_{x_j}a(x)u + a(x)v_j = 0.
	\end{equation}
	
	Let $\rho(x_1)$ be as in the proof of Lemma \ref{lemma1aux}. Multiplying equation \eqref{eq1smoothappendix} by $2\rho(x_1)v_j$ and integrating by parts, we obtain
	\begin{align*}
		\frac{d}{dt}\int_{\mathbb{R}^n} &v_j(x,t)^2 \rho(x_1)\,dx
		+ \int_{\mathbb{R}^n} [3(\partial_{x_1}v_j(x,t))^2 + |\nabla_y v_j(x,t)|^2]\rho'(x_1)\,dx \\
		=& \int_{\mathbb{R}^n} v_j(x,t)^2 \rho'''(x_1)\,dx
		- \int_{\mathbb{R}^n} v_j(x,t)^2 v_1(x,t)\rho(x_1)\,dx
		+ \int_{\mathbb{R}^n} v_j(x,t)^2 u(x,t)\rho'(x_1)\,dx \\
		&-2\int_{\mathbb{R}^n} \partial_{x_j}a(x)u(x,t)v_j(x,t)\rho(x_1)\,dx
		-2\int_{\mathbb{R}^n} a(x)v_j(x,t)^2\rho(x_1)\,dx .
	\end{align*}
	
	Integrating the above equation with respect to time, we obtain
	\begin{align}\label{eq2smoothappendix}
		&\int_0^t\int_{\mathbb{R}^n} |\nabla v_j(x,\tau)|^2 \rho'(x_1)\,dx\,d\tau \nonumber\\
		&\leq \int_{\mathbb{R}^n} v_j(x,0)^2\rho(x_1)\,dx
		+ \int_0^t\int_{\mathbb{R}^n} v_j(x,\tau)^2\rho'''(x_1)\,dx\,d\tau
		- \int_0^t\int_{\mathbb{R}^n} v_j(x,\tau)^2 v_1(x,\tau)\rho(x_1)\,dx\,d\tau \nonumber\\
		&\quad + \int_0^t\int_{\mathbb{R}^n} v_j(x,\tau)^2 u(x,\tau)\rho'(x_1)\,dx\,d\tau
		-2\int_0^t\int_{\mathbb{R}^n} \partial_{x_j}a(x)u(x,\tau)v_j(x,\tau)\rho(x_1)\,dx\,d\tau \nonumber\\
		&:= \int_{\mathbb{R}^n} v_j(x,0)^2\rho(x_1)\,dx + I + II + III + IV .
	\end{align}
	
	Since $\rho'''$ is bounded, we have
	\begin{equation}\label{eq3smoothappendix}
		|I|
		\le
		\|\rho'''\|_{L^\infty}
		\int_0^t\int_{\mathbb{R}^n} v_j(x,\tau)^2\,dx\,d\tau
		\le
		\|\rho'''\|_{L^\infty}
		\|u\|_{L^2([0,T];H^1(\mathbb{R}^n))}^2 .
	\end{equation}
Using H\"older's inequality, Young's inequality, Lemma \ref{VRnL2.3} (or simply the Gagliardo--Nirenberg inequality)\footnote{The use of Lemma \ref{VRnL2.3} would be essential in this step if the gradient were multiplied by $\rho'(x_1)$ instead of $\rho(x_1)$. In that situation we would not need the assumption on the support of $u$. Under the support condition, the standard Gagliardo--Nirenberg inequality also applies.}, estimate \eqref{D3}, and the support hypothesis on $u$, we obtain
\begin{align}\label{eq4smoothappendix}
	|II| \leq&
	\left(
	\int_0^T\int_{\mathbb{R}^n} |v_j(x,\tau)|^3 \rho(x_1)\,dx\,d\tau
	\right)^{\frac{2}{3}}
	\left(
	\int_0^T\int_{\mathbb{R}^n}|v_1(x,\tau)|^3 \rho(x_1)\,dx\,d\tau
	\right)^{\frac{1}{3}}
	\nonumber\\
	\leq&
	\dfrac{2}{3}\int_0^T\int_{\mathbb{R}^n}|v_j(x,\tau)|^3 \rho(x_1)\,dx\,d\tau
	+
	\dfrac{1}{3}\int_0^T\int_{\mathbb{R}^n}|v_1(x,\tau)|^3 \rho(x_1)\,dx\,d\tau
	\nonumber\\
	\leq&
	\varepsilon
	\left\{
	\int_0^T\int_{\mathbb{R}^n} |\nabla v_j(x,\tau)|^2 \rho(x_1)\,dx\,d\tau
	+
	\int_0^T\int_{\mathbb{R}^n} |\nabla v_1(x,\tau)|^2 \rho(x_1)\,dx\,d\tau
	\right\}
	\nonumber\\
	&+
	c(\varepsilon,r,\|\nabla u_0\|_{L^2},\|\rho\|_{L^\infty})
	\nonumber\\
	\leq&
	\varepsilon \|\rho\|_{L^\infty}
	\left\{
	\int_0^T\int_{Q_r} |\nabla v_j(x,\tau)|^2\,dx\,d\tau
	+
	\int_0^T\int_{Q_r} |\nabla v_1(x,\tau)|^2\,dx\,d\tau
	\right\}
	\nonumber\\
	&+
	c(\varepsilon,r,\|u_0\|_{H^1},\|\rho\|_{L^\infty})
	\nonumber\\
	=&
	\varepsilon \|\rho\|_{L^\infty}
	\left[
	\|\nabla v_j\|_{L^2([0,T];Q_r)}^2
	+
	\|\nabla v_1\|_{L^2([0,T];Q_r)}^2
	\right]
	+
	c(\varepsilon,r,\|u_0\|_{H^1},\|\rho\|_{L^\infty}).
\end{align}

By the same argument used in \eqref{eq4smoothappendix}, we obtain the estimate
\begin{align}\label{eq5smoothappendix}
	|III|
	&\le
	\left(
	\int_0^T\int_{\mathbb{R}^n}|v_j(x,\tau)|^3 \rho'(x_1)\,dx\,d\tau
	\right)^{\frac{2}{3}}
	\left(
	\int_0^T\int_{\mathbb{R}^n}|u(x,\tau)|^3 \rho'(x_1)\,dx\,d\tau
	\right)^{\frac{1}{3}}
	\nonumber\\
	&\le
	\varepsilon \|\rho\|_{L^\infty}
	\left\{
	\int_0^T\int_{Q_r} |\nabla v_j(x,\tau)|^2\,dx\,d\tau
	+
	\int_0^T\int_{Q_r} |\nabla u(x,\tau)|^2\,dx\,d\tau
	\right\}
	\nonumber\\
	&+
	c(\varepsilon,r,\|u_0\|_{H^1},\|\rho\|_{L^\infty})
	\nonumber\\
	&\le
	\varepsilon \|\rho\|_{L^\infty}
	\left[
	\|\nabla v_j\|_{L^2([0,T];Q_r)}^2
	+
	\|u\|_{L^2([0,T];H^1(\mathbb{R}^n))}^2
	\right]
	+
	c(\varepsilon,r,\|u_0\|_{H^1},\|\rho\|_{L^\infty}).
\end{align}

Similarly, we obtain
\begin{align}\label{eq6smoothappendix}
	|IV|
	\leq&
	\|a\|_{W^{1,\infty}}
	\left(
	\int_0^T\int_{\mathbb{R}^n} v_j(x,\tau)^2 \rho(x_1)\,dx\,d\tau
	\right)^{\frac12}
	\left(
	\int_0^T\int_{\mathbb{R}^n} u(x,\tau)^2 \rho(x_1)\,dx\,d\tau
	\right)^{\frac12}
	\nonumber\\
	\leq&
	\varepsilon \|\rho\|_{L^\infty}\|a\|_{W^{1,\infty}}
	\|\nabla v_j\|_{L^2([0,T];Q_r)}^2
	+
	c(\varepsilon,r,\|u_0\|_{H^1},\|\rho\|_{L^\infty},\|a\|_{W^{1,\infty}}).
\end{align}

Combining the estimates \eqref{eq2smoothappendix}, \eqref{eq3smoothappendix}, \eqref{eq4smoothappendix}, \eqref{eq5smoothappendix}, and \eqref{eq6smoothappendix}, and choosing $0<\varepsilon\ll1$, we obtain
\begin{equation}\label{eq7smoothappendix}
	\|\nabla v_j\|_{L^2([0,T]\times Q_r)}
	\leq
	c_1 \|u\|_{L^2([0,T];H^1(\mathbb{R}^n))}^2 + c_2,
\end{equation}
where $c_1$ and $c_2$ are positive constants such that
\begin{equation}\label{eq8smoothappendix}
	c_1=c_1(\|\rho\|_{W^{3,\infty}},\|a\|_{W^{1,\infty}})
	\quad\text{and}\quad
	c_2=c_2(\varepsilon,r,\|u_0\|_{H^1},\|\rho\|_{W^{3,\infty}},\|a\|_{W^{1,\infty}}).
\end{equation}

Now we consider the case where $u$ is only assumed to belong to $C([0,T];H^1(\mathbb{R}^n))$. Let us argue that \eqref{eq7smoothappendix} still holds in this case. The argument is standard. 

Consider a sequence $(u_{k,0})\subset H^{\infty}(\mathbb{R}^n)$ approximating $u_0$ in the $H^1(\mathbb{R}^n)$ topology. The well-posedness theory stated above (Theorem \ref{TWP}) provides a sequence of corresponding solutions $(u_k)\subset C([0,T];H^{\infty}(\mathbb{R}^n))$. From \eqref{eq7smoothappendix} we conclude that $(u_k)$ is bounded in $L^2([0,T];H^2(\mathbb{R}^n))$. Hence there exists a function $v\in L^2([0,T];H^2(\mathbb{R}^n))$ such that
\[
u_k \rightharpoonup v
\quad\text{in}\quad
L^2([0,T];H^2(\mathbb{R}^n)).
\]
In particular, since $u_k \rightharpoonup u$ in $L^2([0,T];H^1(\mathbb{R}^n))$, it follows from the uniqueness of the weak limit that $u=v$, and therefore
\[
u\in L^2([0,T];H^2(\mathbb{R}^n)).
\]
Moreover, since each $u_k$ satisfies \eqref{eq7smoothappendix}, that is,
\begin{equation}\label{eq9smoothappendix}
	\|\nabla\partial_j u_k\|_{L^2([0,T]\times Q_r)}
	\le
	c_1\|u_k\|_{L^2([0,T];H^1(\mathbb{R}^n))}^2+c_2,
\end{equation}
taking the $\liminf$ we conclude that the weak limit $u$ also satisfies \eqref{eq7smoothappendix}.

To finish the proof, let $t_0\in(0,T)$. Since $u\in L^2([0,T];H^2(\mathbb{R}^n))$, it follows that $u(t)\in H^2(\mathbb{R}^n)$ for almost every $t\in[0,T]$. Hence there exists $t_\ast\in(0,t_0)$ such that
\[
u(t_\ast)\in H^2(\mathbb{R}^n).
\]

Using the well-posedness theory for the IVP \eqref{eqZK1dampped} with initial data $u(t_\ast)\in H^2(\mathbb{R}^n)$, we conclude by uniqueness that the solution satisfies
\[
u\in C([t_\ast,T];H^2(\mathbb{R}^n)).
\]

In particular,
\[
u\in C([t_0,T];H^2(\mathbb{R}^n)).
\]

\medskip
\noindent{\bf Case $m\ge3$.}
The general case $m\ge3$ follows by induction, differentiating the equation repeatedly and applying the previous step. We omit the details for brevity, as the idea remains the same but the notation becomes cumbersome.
\end{proof}

\end{document}